\documentclass[11pt]{amsart}
\usepackage[utf8]{inputenc}
\usepackage{amscd,amssymb}
\usepackage{amsthm,amsmath,amssymb, mathtools}
\usepackage[matrix,arrow, curve]{xy}
\usepackage{enumerate}

\newtheorem{theorem}{Theorem}[section]

\newtheorem{lemma}[theorem]{Lemma}
\newtheorem{corollary}[theorem]{Corollary}
\newtheorem{conjecture}[theorem]{Conjecture}
\newtheorem*{conjecture*}{Conjecture}

\theoremstyle{definition}

\newtheorem{definition}[theorem]{Definition}
\newtheorem{question}[theorem]{Question}

\theoremstyle{remark}
\newtheorem{remark}[theorem]{Remark}

\makeatletter\@addtoreset{equation}{section} \makeatother
\newcommand{\DF}{{\mathrm{DF}}}
\newcommand{\Amp}{{\mathrm{Amp}}}
\newcommand{\Sesh}{{\mathrm{Sesh}}}
\newcommand{\StabK}{{\mathrm{Amp}}^K}
\newcommand{\rk}{{\mathrm{rk}}}
\newcommand{\Pic}{{\mathrm{Pic}}}
\newcommand{\NE}{{\overline{\mathrm{NE}}}}
\newcommand{\Nef}{{\overline{\mathrm{Nef}}}}

\author{Jesus Martinez-Garcia}
\address{Department of Mathematical Sciences, University of Essex.}
\email{jesus.martinez-garcia@essex.ac.uk}

\title{Constant scalar curvature K\"ahler metrics on  rational surfaces}
\date{15/01/2018}

\begin{document}

\begin{abstract}
We consider projective rational strong Calabi dream surfaces: projective smooth rational surfaces which admit a constant scalar curvature K\"ahler metric for every K\"ahler class. We show that there are only two such rational surfaces, namely the projective plane and the quadric surface. In particular, we show that all rational surfaces other than those two admit a destabilising slope test configuration for some polarization, as introduced by Ross and Thomas. We further show that all Hirzebruch surfaces other than the quadric surface and all rational surfaces with Picard rank $3$ do not admit a constant scalar curvature K\"ahler metric in any K\"ahler class.
\end{abstract}

\dedicatory{In memory of Carmen Garc\'ia Busnadiego}

\sloppy

\maketitle
\tableofcontents
All varieties are assumed to be algebraic, projective and defined over $\mathbb{C}$.

\section{Introduction}
\label{section:intro}
The problem of determining the existence of constant scalar curvature K\"ahler metrics (cscK metrics for short) on projective manifolds is a driving force in complex geometry, which goes back to Calabi's seminal work \cite{Calabi-initial-problem,calabi-extremal-kahler-metrics1}. It has been known for some time that cscK metrics are essentially unique in their K\"ahler class, when they exist \cite{Berman-Berndtsson-uniqueness-extremal, Donaldson-uniqueness-cscK, Chen-Tian-uniqueness-cscK}. This problem has long been expected to have an algebraic formulation, due to the Yau-Tian-Donaldson conjecture:
\begin{conjecture}[{Yau--Tian--Donaldson \cite{TianKEimpliesAnalyticKstability}}]
\label{conjecture:YTD}
Let $X$ be a smooth variety, and let $L$ be an ample line bundle on $X$. Then $X$ admits a constant scalar curvature K\"ahler (cscK) metric in $c_1(L)$ if and only if the pair $(X,L)$ is K-polystable.
\end{conjecture}
It is known in different degrees of generality that K-polystability is a necessary condition for the existence of a cscK metric, with the most general result due to Berman, Darvas and Lu \cite{Berman-Darvas-Lu-csck-implies-K-stability} following work of Darvas and Rubinstein \cite{Darvas-Rubinstein-csck-implies-K-stability}.

Even if Conjecture \ref{conjecture:YTD} holds in all generality, testing which particular polarisations are K-polystable is not an easy task. Moreover, very little is known on how K-stability varies in $\Pic(X)\otimes \mathbb Q$, although in \cite{lebrun-simanca-opennes-k-stability} it was shown that if $\mathrm{Aut}(X)$ is finite, then the set of K-stable polarisations form a (possibly empty) open set. Hence, it is natural to introduce the \emph{set of K-stable polarisations} of a projective manifold $X$:
$$\StabK(X)=\{L\in\Amp(X)\otimes \mathbb Q	\ | \ (X,L) \text{ is K-polystable}\}.$$
In this article we consider the following natural question:
\begin{question}
\label{que:main}
Which manifolds satisfy $\StabK(X)=\Amp^{\mathbb Q}(X)$? Or analogously, which manifolds admit a cscK metric for all ample line bundles?
\end{question}

We answer this question for rational surfaces:
\begin{theorem}
\label{theorem:main}
Given a rational projective surface $S$, $\StabK(S)=\Amp^{\mathbb Q}(S)$ if and only if $S= \mathbb P^2$ or $S=\mathbb P^1\times\mathbb P^1$. 
\end{theorem}
After a first version of this article appeared online \cite{JMG-obstructions-cscK-rational-surfaces-v1}, Chen and Cheng introduced the term \emph{Calabi dream manifolds} to describe those compact manifolds which admit an extremal metric for each K\"ahler class \cite{Chen-Cheng-calabi-dream-manifolds}. Notice that every cscK metric is extremal. Hence the following comes natural:
\begin{definition}
	A K\"ahler manifold is a \emph{strong Calabi dream manifold} if it admits a cscK metric in each K\"ahler class.
\end{definition}
Naturally all strong Calabi dream manifolds are Calabi dream manifolds. Chen and Cheng give some examples of Calabi dream manifolds, which are all in fact strong Calabi dream manifolds \cite{Chen-Cheng-calabi-dream-manifolds}, including K3 surfaces and surfaces of general type with no self-intersection curves. On the other hand, there are examples of surfaces of general type with curves of negative self-intersection which are not strong Calabi dream manifolds \cite{Ross-inventiones}. The name \emph{Calabi dream} refers to the initial belief of Calabi that an extremal metric may exist in each K\"ahler class of all manifolds (see \cite[\S1]{Chen-Cheng-calabi-dream-manifolds} for a short historical account of Calabi's programme).  Theorem \ref{theorem:main} suggests that we can expect few manifolds to be Calabi dream manifolds. Indeed, characterising when such metrics exist (and determining obstructions to existence) is an instigator in the study of extremal and cscK metrics. Notice that since the K\"ahler cone of $\mathbb P^n$ is generated by $\mathrm{Pic}(\mathbb P^n)$, Theorem \ref{theorem:main} provides a complete classification of strong Calabi dream rational projective surfaces, answering a strong analogue (for rational surfaces) of a question in \cite[\S2.4]{Chen-Cheng-calabi-dream-manifolds} for surfaces of general type:
\begin{corollary}
\label{corollary:calabi-dream}
The only projective rational strong Calabi dream surfaces are $\mathbb P^2$ and $\mathbb P^1\times\mathbb P^1$.
\end{corollary}

We introduce the natural counterpart of (strong) Calabi dream manifolds:
\begin{definition}
A compact manifold is a \emph{totally unstable manifold} (a \emph{strongly totally unstable manifold}, respectively) if it does not admit a cscK metric (an extremal metric, respectively) in any of its K\"ahler classes.
\end{definition}

Our algebraic methods can only give a partial answer to the following question:
\begin{question}
\label{que:empty-AmpK}
Which projective rational surfaces are totally unstable?
\end{question}

\begin{theorem}
\label{theorem:low-rank}
Let $\mathbb F_n$ be the $n$-th Hirzebruch surface and let $S\rightarrow \mathbb F_n$ be the blow-up at any point. Then $S$ does not admit a cscK metric in any K\"ahler class. Moreover $\mathbb F_n$ does not admit a cscK metric if $n>0$. In particular, if $S$ is a rational surface with $\rk(\Pic(S))\leqslant 3$ and $S\neq \mathbb P^2$, $\mathbb P^1\times \mathbb P^1$, then $S$ is a totally unstable manifold. Moreover, $\StabK(S)=\emptyset$. 
\end{theorem}
The case of Hirzebruch surfaces seems to have already been known to experts (see \cite[Remark 1]{Fourauthors-Extremal-projective-bundles}) via the Matsushima-Lichn\'erowicz obstruction. We give a new proof using the Ross-Thomas notion of destabilising slope test-configurations \cite{RossThomas1}. On the other hand, the proof when $\rk(\Pic(S))=3$ relies on the Matsushima-Lichn\'erowicz obstruction. In Section \S\ref{section:preliminaries} we recall these constructions and in Section \S\ref{section:proofs} we apply them to prove our results after giving a summary of the geometry of Hirzebruch surfaces.

While Theorem \ref{theorem:low-rank} is not likely to give a complete answer to Question \ref{que:empty-AmpK} for rational surfaces, it is the more general statement we can hope to prove in relation to the rank of the Picard group. Indeed, every del Pezzo surface $S$ with $\mathrm{rk}(\Pic(S))\geqslant 4$ has a K\"ahler-Einstein metric. Furthermore, the blow-up of $S$ at finitely many points has a constant scalar curvature K\"ahler metric by a result of Arezzo and Pacard \cite{Arezzo-Pacard-blow-up}. Nevertheless, the Matsushima-Lich\'erowicz obstruction seems to be too coarse to detect all totally unstable manifolds. Therefore the following question arises naturally:
\begin{question}
\label{que:totally-unstable-aut0}
Is there any totally unstable manifold $X$ rational and with reductive $\mathrm{Aut}^0(X)$?
\end{question}

We expect most manifolds to be at some point in between strongly totally unstable manifolds and strong Calabi dream manifolds, but no \emph{explicit} descriptions of $\StabK(S)$ are known beyond some simple cases such as toric varieties and some partial results on smooth del Pezzo surfaces \cite{Cheltsov-JMG-stable-del-pezzos, Cheltsov-JMG-unstable-del-pezzos}.
\subsubsection*{Acknowledgements:} 
I would like to thank J. Blanc, D. Calderbank, I. Cheltsov, R. Dervan, J. Nordstr\"om, Y. Rubinstein and S. Sun for useful discussions. After the first version of this manuscript appeared on the Arxiv \cite{JMG-obstructions-cscK-rational-surfaces-v1}, and subsequently \cite{Chen-Cheng-calabi-dream-manifolds} introduced Calabi dream manifolds in their study of Conjecture \ref{conjecture:YTD}, X. X. Chen was very kind to engage with me on discussions on the classification of Calabi dream manifolds. Furthermore, Y. Hashimoto explained to me how the initial proof of Theorem \ref{theorem:main} did indeed classify strong Calabi dream rational surfaces. I thank them for those exchanges.

JMG is supported by the Simons Foundation under the Simons Collaboration on Special Holonomy in Geometry, Analysis and Physics (grant \#488631, Johannes Nordstr\"om). 

\section{Destabilising test configurations}
\label{section:preliminaries}

\begin{definition}
\label{definition:TC}
Let $(X, L)$ be a pair formed by a projective manifold and an ample line bundle. Let $\mathbb{G}_m=\mathbb C\setminus\{0\}$ be the multiplicative group. In this article a \emph{test configuration} of $(X,L)$ (with exponent $r$) is a triple $(\mathcal{X},\mathcal{L},{p})$ consisting of
\begin{itemize}
\item a normal projective variety $\mathcal{X}$ with a $\mathbb{G}_m$-action,

\item a flat $\mathbb{G}_m$-equivariant map ${p}\colon\mathcal{X}\to\mathbb{P}^1$ (where $\mathbb{G}_m$ acts naturally on $\mathbb P^1$) such that ${p}^{-1}(t)\cong X$ for every $t\in\mathbb{P}^1\setminus\{0\}$,

\item a $\mathbb{G}_m$-equivariant ${p}$-ample line bundle $\mathcal{L}\to\mathcal{X}$,
such that
$$
\mathcal{L}\Big\vert_{{p}^{-1}(t)}\cong L^{\otimes r}
$$
for every $t\in\mathbb{P}^1\setminus\{0\}$, where we identify ${p}^{-1}(t)$ with $X$.
\end{itemize}
A test configuration $(\mathcal{X},\mathcal{L},{p})$ is a \emph{product test configuration}
if $\mathcal{X}\cong X\times\mathbb{P}^1$ and $\mathcal{L}=p_1^*(L^{\otimes r})$.
A product test configuration is \emph{trivial} if $\mathbb{G}_m$ acts trivially on the left factor of $X\times\mathbb{P}^1$.
\end{definition}

The original definition of test configuration is somewhat different: the fibration $p$ is customarily defined over $\mathbb A^1$ instead of over $\mathbb P^1$. However, Li and Xu showed that the original definition of test configuration compactifies into the one in Definition \ref{definition:TC} (see \cite{Li-Xu-K-stability} or \cite[Section 2]{Cheltsov-JMG-unstable-del-pezzos} for a succinct description of this equivalence). We will use the intersection formula for the generalised Futaki invariant (sometimes known in the literature as Donaldson-Futaki invariant) appearing in \cite[Proposition~6]{Li-Xu-K-stability} and \cite{odaka-test-configurations-restriction,wang-height-git-weight} as the definition. We recall that the \emph{slope} of the pair $(X,L)$ is
$$
\nu(L)=\frac{-K_X\cdot L^{n-1}}{{L}^{n}}.
$$
The \emph{generalised Futaki invariant} of the test configuration $(\mathcal{X},\mathcal{L},{p})$ with exponent $r$
is the number
\begin{equation}
\label{equation:DF-definition}
\DF\big(\mathcal{X},\mathcal{L},{p}\big)=\frac{1}{r^{n}}\Bigg(\frac{n}{n+1}\frac{1}{r}\nu(L)\mathcal{L}^{n+1}+\mathcal{L}^n\cdot\Big(K_{\mathcal{X}}-{p}^*\big(K_{\mathbb{P}^1}\big)\Big)\Bigg),
\end{equation}
where $n$ is the dimension of the variety $X$. If the test configuration $(\mathcal{X},\mathcal{L},{p})$ is trivial then \eqref{equation:DF-definition} gives $\mathrm{DF}(\mathcal{X},\mathcal{L},{p})=0$.

\begin{definition}
\label{definition:K-stability}
The pair $(X,L)$ is \emph{K-polystable} if $\mathrm{DF}(\mathcal{X},\mathcal{L},{p})\geqslant 0$
for every non-trivial test configuration $(\mathcal{X},\mathcal{L},{p})$,
and $\mathrm{DF}(\mathcal{X},\mathcal{L},{p})=0$ only if $(\mathcal{X},\mathcal{L},{p})$ is a product test configuration.
The pair $(X,L)$ is \emph{K-stable} if $\mathrm{DF}\big(\mathcal{X},\mathcal{L},{p}\big)>0$
for every non-trivial test configuration $(\mathcal{X},\mathcal{L},{p})$.
If $\mathrm{DF}(\mathcal{X},\mathcal{L},{p})\geqslant 0$ for every test configuration $(\mathcal{X},\mathcal{L},{p})$,
then $(X,L)$ is \emph{K-semistable}.
\end{definition}

If the pair $(X,L)$ is not K-semistable, then $\mathrm{DF}(\mathcal{X},\mathcal{L},{p})<0$
for some test configuration $(\mathcal{X},\mathcal{L},{p})$ of the pair $(X,L)$.
In this case, we say that $(X,L)$ is \emph{K-unstable}, and $(\mathcal{X},\mathcal{L},{p})$ is a \emph{destabilising} test configuration.

The pair $(X,L)$ is K-polystable (respectively, K-stable or K-semistable) if and only if
the pair $(X,L^{\otimes k})$ is K-polystable (respectively, K-stable or K-semistable) for some positive integer $k$. Thus, we can adapt both Definitions~\ref{definition:TC} and \ref{definition:K-stability} to the case when $L$ is an ample $\mathbb{Q}$-divisor class on the variety $X$ and assume that $r=1$ in the formula \eqref{equation:DF-definition} for the generalised Futaki invariant. This gives us notions of K-polystability, K-stability, K-semistability and K-unstability for varieties polarised by ample $\mathbb{Q}$-divisor classes. In the following we will use $\mathbb Q$-divisor classes and $\mathbb Q$-line bundles interchangeably.

\subsection{Automorphism groups and the Matsushima-Lichn\'erowicz obstruction}
Let us recall the following well-known obstruction to the existence of cscK metrics:
\begin{theorem}[{Matsushima-Lichn\'erowicz's obstruction \cite{lichnerowicz-matsushima-obstruction}, c.f. \cite[\S4.2]{rubinstein-survey}}]
\label{theorem:matsushima-obstruction}
If $X$ is a smooth complex projective variety admitting a cscK metric, then the connected identity component of the automorphism group $\mathrm{Aut}^0(X)$ is reductive.	
\end{theorem}

\begin{remark}
\label{remark:reductive-tricks}
The classification and study of (linearly) reductive linear algebraic groups can become rather technical. Fortunately, for our purposes it will suffice us with the following well known rules (see \cite[Chapter 4]{mukai-book-moduli} for more details):
\begin{enumerate}[(i)]
	\item If $G$ is defined over $\mathbb C$, $G$ is linearly reductive if and only if the connected identity component $G^0$ is reductive.
	\item If $H\triangleleft G$ is normal and $G$ is linearly reductive, then $G/H$ is linearly reductive.
\end{enumerate}
\end{remark}

\begin{lemma}
\label{lemma:reductive-trick}
Let $H$ be a reductive connected complex algebraic linear group. Let $H'\triangleleft H$ be a finite normal subgroup of $H$. The group $G=(\mathbb G_a)^{n+1}\rtimes (H/H')$ is not reductive for $n\geqslant 0$. 
\end{lemma}
\begin{proof}
Notice that $G$ and $H$ are connected and hence $G=G^0$, $H=H^0$ and $G$ is linearly reductive if and only if $G$ is reductive. The group $H/H'$ is linearly reductive by Remark \ref{remark:reductive-tricks} (i) and (ii). The subgroup $H/H'$ is normal in $G$ by assumption. If $G$ was linearly reductive, then $G/H=(\mathbb G_a)^{n+1}$ would be linearly reductive (and hence reductive) but this is well known to be false (e.g. see \cite[Example 4.42]{mukai-book-moduli} for a proof for $\mathbb G_a$ and then apply induction on $n$ using Remark \ref{remark:reductive-tricks}).
\end{proof}

The following well-known result will come useful when applying Theorem \ref{theorem:matsushima-obstruction}.
\begin{lemma}
\label{lemma:automorphism-blowup}
Let $S$ be a compact complex projective surface and $\pi: S'\rightarrow S$ be the blow up of $S$ at a point $p$. All the automorphisms in $\mathrm{Aut}^0 (S')$ leave the exceptional divisor $E =\pi^{-1}(p)$ invariant. Moreover, via restriction to $S'\setminus E$, we have
$$\mathrm{Aut}^0 (S')\cong\Big\{\phi\in \mathrm{Aut}^0 (S) \ \Big| \ \phi(p) = p\Big\} =: \mathrm{Aut}^0 (S,p).$$
\end{lemma}
\begin{proof}
An automorphism $\sigma\in \mathrm{Aut}^0(S')$ must leave the exceptional divisor $E$ invariant and hence, it induces an automorphism $\pi_*\circ\sigma\in \mathrm{Aut}^0(S,p)$ fixing $p$. Conversely, let $\phi\colon S\rightarrow S$ be an automorphism fixing $p$. The inverse image in $S'$ via $\pi$ of both $p$ and $\phi(p)$ is a Cartier divisor. Hence, the universal property of the blow-ups \cite[Corollary 7.15]{HartshorneAG} induces a unique automorphism $\widetilde \phi\colon S'\rightarrow S'$  such that $\pi\circ \phi=\widetilde \phi\circ \pi$ and in particular $E$ is $\widetilde\phi$-invariant.
\end{proof}

\subsection{Slope stability}
\label{section:slope}
Let us recall the construction of \emph{slope test configurations}, introduced by Ross and Thomas \cite{RossThomas1} (also known in the literature as deformation to the normal cone). Our notation follows our previous work with Cheltsov \cite{Cheltsov-JMG-unstable-del-pezzos}. We have simplified the hypothesis to the needs of the problem. See \cite{Cheltsov-JMG-unstable-del-pezzos} or \cite{RossThomas1} for a more general treatment.

Let $S$ be a smooth surface, $L$ be an ample $\mathbb Q$-divisor class of $S$ and $Z$ be a smooth irreducible divisor in $S$.
\begin{definition}
The \emph{Seshadri constant} of $L$ at $Z$ is the real number
$$\Sesh(S, L, Z)=\sup\left\{\lambda \ : \ L -Z\text{ is nef}\right\}.$$
\end{definition}
Let $\pi_Z\colon \mathcal X\rightarrow S\times\mathbb P^1$ be the blow-up of $S\times \mathbb P^1$ along $Z\times \{0\}$ with exceptional divisor $E_Z$. By a slight abuse of notation we identify $Z\subset S$ with $Z\times \{0\}\in S\times\mathbb P^1$. Let $p_{\mathbb P^1}\colon S\times \mathbb P^1 \rightarrow \mathbb P^1$ and $p_{S}\colon S\times \mathbb P^1\rightarrow S$ be the natural projections. Let $p=p_{\mathbb P^1}\circ \pi_Z$ and define the $\mathbb Q$-divisor
$$\mathcal L_{\lambda}:=(p_S \circ \pi_Z)^*L-\lambda E_Z.$$
We call $(\mathcal X, \mathcal L_{\lambda})$ the \emph{slope test configuration} of $(S, L)$ centred at $Z$.

\begin{lemma}[{\cite{RossThomas1}, c.f. \cite[Lemma 2.2]{cheltsov-rubinstein-flops}, \cite[Lemma 3.1]{Cheltsov-JMG-unstable-del-pezzos}}]
\label{lemma:DF-slope}
The $\mathbb Q$-divisor class $\mathcal L_{\lambda}$ is $p$-ample for all rational $0<\lambda<\Sesh(S,L,Z)$. Moreover its generalised Futaki invariant satisfies
\begin{equation*}
\mathrm{DF}(\mathcal X, \mathcal L_{\lambda})=\frac{2}{3}\nu(L)\left[-3\lambda^2L\cdot Z + \lambda^3Z^2\right] + \lambda^2(2-2g(Z)) + 2\lambda L \cdot Z,
\end{equation*}
where $g(Z)$ is the genus of $Z$.
\end{lemma}

The advantage of the Ross-Thomas construction of test configurations is that it allows us to extend destabilising test-configurations of a pair $(S, L)$ to destabilising test configurations of pairs $(S', L')$ where $S'\rightarrow S$ is \emph{any} composition of blow-downs of $(-1)$-curves not supported on the curve $Z\subset S$ and $L'$ is \emph{some} polarisation. This is the content of \cite[Corollary~5.29]{RossThomas1}. We give a more detailed proof of the result for the convenience of the reader, since this is our main tool to prove Theorem \ref{theorem:main} and we will need a precise statement.

Suppose that $(S,L)$ is destabilised by a slope test configuration $(\mathcal X, \mathcal L_\lambda)$ centred at a smooth irreducible curve $Z\subset S$ and let $p\in S\setminus Z$ be a point. Let $g\colon S^{\prime}\to S$ be the blow-up of $p$ and $G$ be its exceptional curve.  Denote by $Z^{\prime}$ the proper transform of $Z$ in $S^\prime$. The morphism $g$ induces a blow-up $h:S'\times \mathbb P^1\rightarrow S\times \mathbb P^1$ of $p\times \mathbb P^1  \subset S\times \mathbb P^1$. There exists a commutative diagram
$$
\xymatrix{
\mathcal{X}^\prime\ar@/_1pc/@{->}[dddr]_{p^\prime} \ar@{->}[rrrr]^{f} \ar@{->}[ddr]^{\pi_{Z'}}&&&&\mathcal X\ar@{->}[ddl]_{\pi_{Z}}\ar@/^1pc/@{->}[dddl]^{p}  \\
&S^\prime \ar@{->}[rr]^{g}&&S &\\
&S^\prime\times\mathbb P^1\ar@{->}[u]_{q_{S^\prime}}\ar@{->}[d]^{q_{\mathbb{P}^1}^\prime}\ar@{->}[rr]^{h}&&S\times \mathbb P^1\ar@{->}[u]^{q_{S}}\ar@{->}[d]_{q_{\mathbb{P}^1}} & \\
&\mathbb{P}^1\ar@{=}[rr]&&\mathbb{P}^1,&}
$$
where $q_S, q_{S'}, q_{\mathbb P^1}, q'_{\mathbb P^1}$ are the obvious projections, $\pi_{Z'}\colon\mathcal X'\rightarrow S'\times \mathbb P^1 $ is the blow-up of $Z'\subset S'\times \{0\}$ and $f\colon \mathcal X'\rightarrow \mathcal X$ is the contraction of $G\times \mathbb P^1\subset \mathcal X'$. We denote the exceptional divisor of $\pi_{Z'}$ by $E_{Z'}$.

We may choose some sufficiently small positive rational number $\varepsilon>0$ so that $L^\prime_\varepsilon=g^*(L)-\varepsilon G$ is ample. By Lemma \ref{lemma:DF-slope}, we may choose a positive rational number $\lambda=\lambda(\varepsilon)$ such that $0< \lambda<\Sesh(S^\prime,L^\prime_\varepsilon,Z^\prime)$ and then $\mathcal{L}_{\lambda}^\prime=(p_{S^\prime}\circ\pi_{Z^\prime})^*(L^\prime)-\lambda E_{Z^\prime}$ is $p'$-ample where $p'=q'_{\mathbb P^1}\circ \pi_{Z'}$, $(\mathcal{X}^\prime,\mathcal{L}_{\lambda}^\prime,p^\prime)$ is a test configuration of $(S^\prime,L^\prime_\epsilon)$ (in fact its slope test configuration centred at $Z'$) and 
$$
\mathrm{DF}\big(\mathcal{X}^\prime,\mathcal{L}_{\lambda}^\prime,p^\prime\big)=\frac{2}{3}\nu(L^\prime_\varepsilon)\Big(\lambda^3Z^2-3\lambda^2L\cdot Z\Big)+\lambda^2\big(2-2g(Z)\big)+2\lambda L\cdot Z,
$$
where we use the fact that $p\not\in Z$. The latter also implies that 
$\displaystyle{\lim_{\varepsilon\to 0^+}}\Sesh\big(S^\prime,L^\prime_\varepsilon,Z^\prime\big)=\Sesh(S,L,Z)$ and
$\nu(L^\prime_\varepsilon)=\frac{-K_{S^\prime}\cdot L^\prime_\varepsilon}{L^\prime_\varepsilon\cdot L^\prime_\varepsilon}=\frac{-K_S\cdot L-\varepsilon}{L^2-\varepsilon^2}$ so that $\displaystyle{\lim_{\varepsilon\to 0^+}\nu(L^\prime_\varepsilon})=\nu(L)$. As a result $\displaystyle{\lim_{\varepsilon\to 0^+}}\mathrm{DF}\big(\mathcal{X}^\prime,\mathcal{L}_{\lambda}^\prime,p^\prime\big)=\mathrm{DF}\big(\mathcal{X},\mathcal{L}_{\lambda},p\big)$. In summary, since for fixed $\lambda$ the generalised Futaki invariant $\mathrm{DF}\big(\mathcal{X}^\prime,\mathcal{L}_{\lambda}^\prime,p^\prime\big)$ is a quotient of polynomials on $\epsilon$ (and hence it is continuous) we have proved:
\begin{lemma}[{\cite[Corollary~5.29]{RossThomas1}}]
\label{lemma:RT-blow-up}
Suppose that $0<\lambda<\Sesh(S,L,Z)$ and $\mathrm{DF}\big(\mathcal{X},\mathcal{L}_{\lambda},p)<0$.
Then $\lambda<\Sesh(S^\prime,L^\prime_\varepsilon,Z^\prime)$ and $\mathrm{DF}(\mathcal{X}^\prime,\mathcal{L}_{\lambda}^\prime,p^\prime)<0$
for sufficiently small $\varepsilon>0$.
\end{lemma}

We conclude this section by posing the following conjecture:
\begin{conjecture}
\label{conjecture:mine}
Let $(X, L)$ be a  pair destabilised by a test configuration $(\mathcal X, \mathcal L)$ and let $\pi\colon X'\rightarrow X$ be a projective birational morphism. There is an ample line bundle $L'$ of $X'$, a destabilizing test configuration $(\mathcal X', \mathcal L')$ of $(X', L')$, a K-unstable pair $(X'', L'')$ destabilized by a test configuration $(\mathcal X'', \mathcal L'')$, a birational map $\phi\colon X\dashrightarrow X''$, a birational morphism $\psi:X'\rightarrow X''$ satisfying $\psi=\phi\circ\pi$ and $\psi(L')=L''$, and morphisms $f\colon \mathcal X'\rightarrow \mathcal X''$, $g\colon \mathcal X'\rightarrow \mathcal X$ such that $\mathcal L''=f_*(\mathcal L')$, $\mathcal L''=g_*(\mathcal L)$, $f|_F=\psi$ and $g|_F=\phi$, where $F$ is the general fibre of $\mathcal X'$.
\end{conjecture}
There is some additional evidence to support Conjecture \ref{conjecture:mine} beyond Lemma \ref{lemma:RT-blow-up}. For instance, it follows from \cite[Corollary 3.9]{Cheltsov-JMG-unstable-del-pezzos} that Conjecture \ref{conjecture:mine} holds for the \emph{flop slope} test configurations introduced by Cheltsov and Rubinstein in \cite{cheltsov-rubinstein-flops} when $\pi$ is not supported on the slope curve.
\section{Proofs}
\label{section:proofs}
\subsection{Hirzebruch surfaces}
We recall the basic geometry, positivity, equations and intersection theory of Hirzebruch surfaces. See \cite[Chapter IV]{Beauville1996} and \cite[Chapter V.2]{HartshorneAG} for the details. Denote by $\mathbb F_n:=\mathbb P(\mathcal O_{\mathbb P^1}\otimes\mathcal O_{\mathbb P^1}(n))\rightarrow \mathbb P^1$ the unique $n$-th Hirzebruch surface, where $n\in \mathbb Z_{\geqslant 0}$. This is the unique rational ruled surface whose Picard group is isomorphic to $\mathbb Z\oplus\mathbb Z$ and which contains a unique smooth rational curve of self-intersection $-n$, which is unique if $n>0$. We denote this curve by $Z_n$. We denote by $F$ the class of a fibre of the natural projection $\mathbb F_n\rightarrow \mathbb P^1$. We have that
$$Z_n\cdot F=1,\qquad Z_n^2=-n,\qquad F^2=0.$$
The Mori cone of effective curves $\NE(\mathbb F_n)$ is two dimensional and generated by $F$ and $Z_n$. 
In the special case where $n=0$, we have $\mathbb F_0\cong \mathbb P^1\times\mathbb P^1$ and the classes $F$ and $Z_0$ are the class of a fibre of each of the two different natural projections to $\mathbb P^1$. Moreover, we have
$$-K_{\mathbb F_n}\sim 2Z_n+(n+2)F,\qquad -K_{\mathbb F_n}\cdot F=2,\qquad -K_{\mathbb F_n}\cdot Z_n=2-n.$$
The Nef cone $\Nef(\mathbb F_n)$ is generated by $F$ and $aZ_n+(na)F$, i.e. $C\sim_{\mathbb Q}aZ_n+bF$ is ample if and only if $a>0$ and $b>na$.

We can also work with $\mathbb F_n$ using coordinates. Indeed, its loci is given by
$$\mathbb F_n=\{([x:y:z], [u:v])\in \mathbb P^2\times\mathbb P^1 \ | \ yv^n=zu^n\}.$$
In this coordinates the curve $Z_n$ is given by
$$Z_n=\{([1:0:0],[u:v]) \ | \ (u:v)\in \mathbb P^1\}\subset \mathbb F_n,$$
and the map $\mathbb F_n\rightarrow \mathbb P^1$ is given by projection on the second factor.

Given any rational surface $S\neq \mathbb P^2$, there is a morphism $S\rightarrow \mathbb F_n$ for some $n\geqslant 0$, which factors as the contraction of $k\geqslant 0$ $(-1)$-curves. In addition, there is a morphism $\mathbb F_1\rightarrow \mathbb P^2$ contracting $Z_1$. There is a commutative diagram
\begin{equation}
\label{eq:Fn-triangle}\begin{gathered}{\xymatrix{ & Y\ar@{->}[dl]_{\sigma_n}\ar@{->}[dr]^{\sigma_{n+1}} &\\
\mathbb F_n\ar@{-->}[rr]^{\phi_n}& &\mathbb F_{n+1},}
}\end{gathered}\end{equation}
where $\sigma_n$ is the blow-up of the point $p_n=([1:0:0],[1:0])\in \mathbb F_n$ and $\sigma_{n+1}$ is the contraction of the proper transform in $Y$ of the fibre of $\mathbb F_n\rightarrow \mathbb P^1$ passing through $p_n$. Conversely, up to isomorphism, $\sigma_{n+1}$ is the blow-up of $q_{n+1}=([0:0:1], [0:1])\in \mathbb F_{n+1}$ and $\sigma_n$ is the contraction of the proper transform in $Y$ of the fibre of $\mathbb F_{n+1}\rightarrow \mathbb P^1$ passing through $q_{n+1}$.

We can identify the additive group with the group of homogeneous polynomials of degree $n$ in variables $z_0,z_1$, i.e.  $(\mathbb G_a)^{n+1}\cong \mathbb C^{n+1}\cong \mathbb C[z_0,z_1]_n$ under this identification. We define the group homomorphism $\phi\colon\mathrm{GL}(2, \mathbb C)\rightarrow \mathrm{Aut}(\mathbb C^{n+1})$ given by 
$$\phi(M)(p(z_0,z_1))=p\left(M\cdot \left(z_0,z_1\right)^t\right).$$
Using $\phi$, we define the semi-direct product $G_n\coloneqq(\mathbb G_a)^{n+1}\rtimes \mathrm{GL}(2,\mathbb C)$
with product rule
$$(p,M)\cdot(q,N)=(q+p(N\cdot(z_0,z_1)^t), M\cdot N).$$

Let $n\geqslant 1$. The group $G_n$ acts on $\mathbb F_n$. Indeed, the element
$$\left(a_0z_0^n+a_1z_1z_0^{n-1}+\cdots +a_nz_1^n,\begin{pmatrix}
	a & b\\
	c & d
\end{pmatrix}\right)$$
acts on the point $([x:y:z],[u:v])\in \mathbb F_n$ by sending it to 
\begin{equation}\label{eq:action-coordinates}
\begin{gathered}
\begin{cases}
\left(\left[xu^n+y(a_0u^n+\cdots+a_nv^n):y(au+bv)^n:y(cu+dv)^n\right],\left[au+bv: cu+dv\right]\right), \text{ if }u\neq 0,\\
\left(\left[xv^n+z(a_0u^n+\cdots+a_nv^n):z(au+bv)^n:z(cu+dv)^n\right],\left[au+bv: cu+dv\right]\right), \text{ if }v\neq 0.
\end{cases}
\end{gathered}
\end{equation}
Let $\mu_n\triangleleft \mathrm{GL}(2,\mathbb C)$ be the finite subgroup consisting of the diagonal matrices $A=\lambda I$ where $\lambda$ is an $n$-th root of unity. We have an exact sequence of group homomorphisms
$$1\longrightarrow \mu_n\longrightarrow G_n\longrightarrow (\mathbb G_a)^{n+1}\rtimes(\mathrm{GL}(2,\mathbb C)/\mu_n)\rightarrow 1,$$
and it is easy to check that $\mu_n$ is the kernel of the action $G_n$ on $\mathbb F_n$. Hence $(\mathbb G_a)^{n+1}\rtimes(\mathrm{GL}(2,\mathbb C)/\mu_n)\subseteq\mathrm{Aut}(\mathbb F_n)$. While the above description is not entirely necessary to prove the following lemma, it will come useful when studying the automorphism groups in the blow-ups of $\mathbb F_n$.
\begin{lemma}
\label{lemma:hirzebruch-obstruction}
Let $S=\mathbb F_n$ for $n\geqslant 1$. Then $\mathrm{Aut}(S)=\mathrm{Aut}^0(S)\cong (\mathbb G_a)^{n+1}\rtimes (\mathrm{GL}(2,\mathbb C)/\mu_n)$. In particular, $\mathrm{Aut}^0(S)$ is not reductive and $S$ does not admit a cscK metric in any K\"ahler class.
\end{lemma}
\begin{proof}
The contraction of $Z_n$ induces a morphism $\pi\colon \mathbb F_n\rightarrow S'=\mathbb P(1,1,n)$. Suppose $n\geqslant 2$. Then $p=\pi(Z_n)=[0:0:1]$ is the unique singular point of $S'$ and $p$ must be fixed by $\mathrm{Aut}(S')$. The group $\mathrm{Aut}(S)$ acts transitively on the fibres $\mathbb F_n\rightarrow \mathbb P^1$ and it fixes $Z_n$, since it is the unique curve in $S$ of negative self-intersection. Hence, by Lemma \ref{lemma:automorphism-blowup} $\mathrm{Aut}(S)\cong \mathrm{Aut}(S')$. Any $\phi\in \mathrm{Aut}(S')$ can be given by
$$[t_0:t_1:t_2]\mapsto[at_0+bt_1\colon ct_0+dt_1\colon et_2+p(t_0:t_1)]$$
where $p$ is a homogeneous polynomial of degree $n$. Observe that $(\mathbb G_a)^{n+1}\triangleleft \mathrm{Aut}(S')$ is identified with the subgroup of automorphisms given by
$$[t_0:t_1:t_2]\mapsto[t_0\colon t_1\colon t_2+p(t_0:t_1)]$$
and $\mathrm{Aut}(S')/(\mathbb G_a)^{n+1}\cong H/\mathbb G_m$, where $H$ is given by automorphisms of type
$$[t_0:t_1:t_2]\mapsto[at_0+bt_1\colon ct_0+dt_1\colon et_2]$$
and the subgroup $\mathbb G_m$ consists of diagonal automorphisms $[t_0:t_1:t_2]\mapsto[a t_0\colon a t_1\colon at_2]$.
Hence $H/\mathbb G_m\cong \mathrm{GL}(2,\mathbb C)/\mu_m$ and
$$\mathrm{Aut}(S)\cong \mathrm{Aut}(S')\cong (\mathbb G_a)^{n+1}\rtimes (\mathrm{GL}(2,\mathbb C)/\mu_n)\cong\mathrm{Aut}^0(S),$$
since the group is connected. Now suppose $\pi\colon S=\mathbb F_n\rightarrow S'=\mathbb P^2$ is the blow-up of $p=[0:0:1]\in \mathbb P^2$. By Lemma \ref{lemma:automorphism-blowup},
$$\mathrm{Aut}(\mathbb F_1)\cong \mathrm{Aut}(\mathbb P^2, p)=\left\{\begin{pmatrix}
	1 & a_0 & a_1\\
	0 & a 	& b\\
	0 & c		& d
\end{pmatrix}\in \mathrm{GL}(3,\mathbb C)\right\}.$$
We have a group isomorphism $\mathrm{Aut}(\mathbb F_1)\rightarrow (\mathbb G_a)^2\rtimes \mathrm{GL}(2,\mathbb C)$, given by
$$\begin{pmatrix}
	1 & a_0 & a_1\\
	0 & a 	& b\\
	0 & c		& d
\end{pmatrix} \mapsto \left(az_0+a_1z_1, \begin{pmatrix}
a 	& b\\
c		& d
\end{pmatrix}\right).$$
Non-reductivity of $\mathrm{Aut}(\mathbb F_1)$ follows from Lemma \ref{lemma:reductive-trick} and the fact that $\mathrm{GL}(2,\mathbb C)$ is linearly reductive \cite[Corollary 4.44]{mukai-book-moduli}.
\end{proof}
The last lemma and the following one give a complete answer to the existence of cscK metrics on Hirzebruch surfaces.
\begin{lemma}
\label{lemma:k-polystable-minimal}
Let $S\cong \mathbb P^2$ or $S\cong \mathbb P^1\times\mathbb P^1$. Then $S$ admits a cscK metric in any K\"ahler class. In particular $(S, L)$ is K-polystable for all $\mathbb Q$-ample line bundles $L$. 
\end{lemma}
\begin{proof}
The Picard group of $\mathbb P^n$ has $\mathrm{rk}(\mathrm{Pic}(\mathbb P^n))=1$ for all $n\geqslant 1$ and the K\"ahler cone is one-dimensional. Therefore any ample line bundle $L=\mathcal O_{\mathbb P^n}(aH)$ for some $a\in \mathbb N$, where $H$ is the class of a hyperplane. The Fubini-Study metric $g_{FS}$ on $\mathbb P^n$ is a K\"ahler--Einstein metric and therefore a cscK metric in $c_1(3H)$. Hence, $\frac{a}{3}g_{FS}$ is a cscK metric in $c_1(L)$. In particular this applies to $\mathbb P^2$. If $L\cong  p_1^*\mathcal O_{\mathbb P^1}(aH)\otimes p_2^*\mathcal O_{\mathbb P^1}(bH)$ is ample and $g_a$, $g_b$ are cscK metrics in $c_1(\mathcal O_{\mathbb P^1}(aH))$ and $c_1(\mathcal O_{\mathbb P^1}(bH))$, then $g_a+ g_b$ is a cscK metric in $c_1(\mathcal O_{\mathbb P^1\times \mathbb P^1}(L))$. Since an $L$-polarised surface $S$ admitting a cscK metric is K-polystable \cite{Berman-Darvas-Lu-csck-implies-K-stability, Darvas-Rubinstein-csck-implies-K-stability} and K-polystability is invariant under scaling of $L$, the lemma follows.
\end{proof}
Of course, the above lemma is well known and it can be proved in many less direct ways. For instance, since they are toric surfaces, one can use Donaldson's theory \cite{Donaldson-toric-surfaces-K-stability, Donaldson-scalar-curvature-K-stability-toric}. However, we include the above proof since it is one of the few instances where one could give an explicit description of the metric.

\subsection{Proof of Theorem \ref{theorem:main}}
While the previous section completes the classification for Hirzebruch surfaces, it does not say anything of an arbitrary rational surface. For this reason we take the following alternative take to Lemma \ref{lemma:hirzebruch-obstruction} via destabilising test configurations:
\begin{lemma}
\label{lemma:Hirzebruch}
The pair $(\mathbb F_n, L)$, where $n\geqslant 1$ and $L$ is any ample $\mathbb Q$-divisor is destabilised by the slope test configuration $(\mathcal X, \mathcal L_\lambda)$ of $(\mathbb F_n, L)$ centred at $Z_n$.
\end{lemma}
\begin{proof}
Let $L=aZ_n+bF$ be an ample $\mathbb Q$-divisor class of $S=\mathbb F_n$. In particular $a>0$ and $b>na>0$ and 
$$\nu(L)=\frac{-K_S\cdot L}{L^2}=\frac{(2-n)a+2b}{2ab-a^2n}.$$
Let $Z=Z_n$. Then we have
$$L\cdot Z=b-na,\qquad Z^2=-n,\qquad L-\lambda Z\equiv (a-\lambda)Z+bF,$$
and $L-\lambda Z$ is ample if and only if $a>\lambda>\frac{na-b}{n}$, so $\Sesh(X, L,Z)=a$.

Let $(\mathcal X, \mathcal L_\lambda)$ be the slope test configuration of $(X,L)$ centred at $Z$. Substituting  with $L, Z$ and $\lambda=\Sesh(X,L,Z)=a$ in Lemma \ref{lemma:DF-slope} we obtain
\begin{align*}
\DF(\mathcal X, \mathcal L_a )	&=\frac{2}{3}\cdot\frac{(2-n)a+2b}{2ab-na^2}\Big(-3a^2(b-na)-a^3n\Big)+2a^2+2a(b-na)\\
										&=\frac{2a}{3}\cdot\frac{-6ab+4a^2n+7nab-2n^2a^2-6b^2}{2b-na}+2a\frac{(a+b-na)\cdot(6b-3na)}{3(2b-na)}\\
										&=\frac{2a}{3}\cdot\frac{-6ab+4a^2n+7nab-2n^2a^2-6b^2}{2b-na}+2a\frac{6ab+6b^2-9nab-3na^2+3n^2a^2}{3(2b-na)}\\
										&=\frac{2a^2n}{3}\cdot \frac{(a-2b+na)}{2b-na}\\
										&<\frac{2a^2n}{3}\cdot \frac{(a-na)}{2b-na}\\
										&=\frac{2a^3n}{3}\cdot \frac{(1-n)}{2b-na}\leqslant 0,
\end{align*}
as $b>na>0$ and $n\geqslant 1$. Hence $\DF(\mathcal X, \mathcal L_a )<0$ for all $n\geqslant 1$. Since $\DF(\mathcal X, \mathcal L_\lambda )$ is a continuous function on $\lambda$, for some $0<\lambda<\Sesh(S,L,Z)=a$, we have $\DF(\mathcal X, \mathcal L_\lambda )<0$ and $(S, L)$ is K-unstable.
\end{proof}

\begin{remark}
After circulating a version of this article, D. Calderbank let the author know that in \cite{Fourauthors-Extremal-metrics-stability} a similar strategy is followed (namely the use of a destabilising slope test configuration) to show the obstruction to the existence of extremal metrics on certain ruled surfaces, generalising a technique of Sz\'ekelyhidi \cite{szekelyhidi-extremal-metrics-and-k-stability}.
\end{remark}

\begin{proof}[Proof of Theorem \ref{theorem:main}]
Suppose $S\neq \mathbb P^2$ and $S\neq \mathbb P^1\times \mathbb P^1$. We need to find an ample $\mathbb Q$-line bundle $L$ such that $(S,L)$ is K-unstable. We will proceed by induction on the rank $k=\mathrm{rk}(\mathrm{Pic}(S))\geqslant 2$. By the classification of rational surfaces we have that there is a morphism
$f\colon S=S_k\rightarrow S_{k-1}\rightarrow \cdots \rightarrow S_2\eqqcolon\overline S=\mathbb F_n$ which is a composition of blow-ups $S_i\rightarrow S_{i-1}$ where $\mathrm{rk}(\mathrm{Pic}(S_i))=i$ (or equivalently a composition of contractions of $(-1)$-curves). We may assume that $n\geqslant 1$. Indeed, if $n=0$, then the last blow-up is $f_2\colon S_3\rightarrow S_2=\mathbb P^1\times\mathbb P^1$ where $S_3$ is the smooth del Pezzo surface of degree $7$ which has three $(-1)$-curves $E_1, E_2, E_3$ where $E_2\cdot E_3=E_2\cdot E_1=1$ and $E_1\cdot E_3=0$ and $f_2$ is the contraction of $E_2$. However, we may replace $f_2$ by the contraction of $E_1$ and then $f_2\colon S_3\rightarrow \mathbb F_1$. 

If $k=2$, then $f$ is an isomorphism and the proof follows from Lemma \ref{lemma:Hirzebruch}. Suppose $k=3$ and $f\colon S=S_3\rightarrow \mathbb F_n$ is the blow-up at a point $p$. If $p\not\in Z_n\subset \mathbb F_n$, then the proof follows from lemmas \ref{lemma:Hirzebruch} and \ref{lemma:RT-blow-up}. If $p\in Z_n$, let $E\subset S_k$ be the $f$-exceptional curve and $F$ be the proper transform of the unique fibre of $\mathbb F_n\rightarrow \mathbb P^1$. Then $F$ is a $(-1)$-curve and its contraction gives a morphism $f'\colon S=S_3\rightarrow \mathbb F_{n+1}$ which is the blow-up of a point $q\not\in Z_{n+1}\subset \mathbb F_{n+1}$ and then the proof is as in the case $p\not\in Z_n$. In particular, we have proven that for $k=3$, there is a morphism $f\colon S_k\rightarrow \mathbb F_{n}$ which is an isomorphism around $Z_n$, and there is an ample $\mathbb Q$-line bundle $L$ in $S_k$ such that $(S,L)$ is destabilised by the slope test configuration centred at the proper transform of $Z_n$.

Now, for the induction step, we suppose that we have an ample $\mathbb Q$-line bundle $L$ on $S_{k-1}$ and a composition of blow-ups $h\colon S_{k-1}\rightarrow \mathbb F_n$ such that $(S_{k-1}, L)$ is destabilised by the slope test configuration centred at the proper transform of $Z_n$ in $S_{k-1}$ and that $h$ is an isomorphism around $Z_{n}$. We let $f\colon S=S_k\rightarrow S_2=\mathbb F_n$ factor as $f=h\circ \pi$, where $h\colon S_{k-1}\rightarrow \mathbb F_n$ is as in the induction hypothesis and $\pi\colon S_k\rightarrow S_{k-1}$ is the blow-up at a point $p\in S_{k-1}$ with exceptional divisor $E\subset S_k$. Let $Z_n'$ be the proper transform of $Z_n$ in $S_{k-1}$ via $h$.
If $p\not\in Z_n'$, then the result follows from Lemma \ref{lemma:RT-blow-up}.

Hence, suppose that $p\in Z_n'$ and let $F_{k-1}$ (respectively $F_k$) be the proper transform in $S_{k-1}$ (respectively $S_k$) of the fibre $F$ of $\mathbb F_n\rightarrow \mathbb P^1$ passing through $h(p)$. Notice that $F_k^2=F_{k-1}^2-1$. Denote by $C_i$ (respectively $\widetilde C_i$) the proper transform in $S_{k-1}$ (respectively $S_k$) of the exceptional divisor of $S_i\rightarrow S_{i-1}$ for $i=3,\cdots , k-1$. Let $F_i$ be the proper transform of $F$ in $S_i$ and let $l$ be the smallest index such that $F_l^2=-1$. Observe that if $l=k$ then $h$ is an isomorphism near $Z_n$ and $F$. By the induction hypothesis $\widetilde C_i^2=C_i^2$. The latter allows us to define the morphism $g\colon S_k\rightarrow S'$  to a smooth surface $S'$ as the successive contraction of $\widetilde C_k, \ldots, \widetilde C_{l+1}$, where we define $g$ to be the identity morphism if $k=l$. Then $(g(F_k))^2=(g(E))^2=-1$. Let $\pi'\colon S'\rightarrow S''$ be the contraction of $g(F_k)$ to some smooth surface $S''$. Let $Z_n''=(\pi'\circ g)(\widetilde Z_n)\subset S''$. Notice that $(\widetilde Z_n)^2=(Z_n'')^2=-n-1$.

Hence, the composition $\pi'\circ g$ is an isomorphism around $\widetilde Z_n'$. Let $E''=(\pi'\circ g)(E)$, $C_l''=(\pi'\circ g)(\widetilde C_l), \ldots, C_3''=(\pi'\circ g)(\widetilde C_3)$. By the inductive hypothesis $E\cdot C_i''=C_i''\cdot Z_n''=0$, for $i=3, \ldots, l$. Let $g'\colon S''\rightarrow S'''$ be the successive contraction of $C_l'', \ldots, C_3''$. Hence $(g'(Z_n''))^2=(Z_n'')^2=-n-1$ and $\rk(\Pic)(S''')=2$. Hence $S'''\cong\mathbb F_{n+1}$, $g'(Z_n'')=Z_{n+1}$ and $g'\circ\pi'\circ g\colon S_k\rightarrow \mathbb F_{n+1}$ is an isomorphism around $Z_{n+1}$, completing the proof of the inductive statement. In particular, given any rational surface $S$, we deduce that there is a morphism $S\rightarrow \mathbb F_{m}$ for some $m\geqslant 1$ such that $S$ is an isomorphism around $Z_m$, and an ample $\mathbb Q$-line bundle $L$ on $S$ such that,  by means of Lemma \ref{lemma:RT-blow-up}, we construct a destabilising slope test configuration for $(S, L)$. The result for $S=\mathbb P^1\times\mathbb P^1$ and $S=\mathbb P^2$ follows from Lemma \ref{lemma:k-polystable-minimal} and the fact that a pair $(S, L)$ admitting a cscK metric is K-polystable \cite{Berman-Darvas-Lu-csck-implies-K-stability, Darvas-Rubinstein-csck-implies-K-stability}.
\end{proof}

\begin{remark}
If Conjecture \ref{conjecture:mine} holds, we may expect a similar approach to answer Question \ref{que:main} for other birational classes of surfaces as the one presented in the last proof. Namely, given a pair $(S, L)$, we may apply the Minimal Model Programme to find a morphism $S\rightarrow S_0$ where $S_0$ is a smooth surface with no $(-1)$-curves. Then we may classify all surfaces $S_0$ with no $(-1)$-curves such that $\StabK(S_0)\neq \Amp^{\mathbb Q}(S_0)$ and apply a solution to Conjecture \ref{conjecture:mine}. For all those surfaces with $\StabK(S_0)=\Amp^{\mathbb Q}(S_0)$, we may attempt to show that if enough (infinitely closed) points in $S_0$ are blown up, we may end-up with a morphism $g\colon S'\rightarrow S_0$ such that $\StabK(S')\neq \Amp^{\mathbb Q}(S')$. Ideally, we should be able to describe the \emph{smallest} such $S'$ (in the case of rational surfaces $g$ is the identity, except for the special cases of $\mathbb F_1$ and $\mathbb F_0$). While a similar approach may also be considered in higher dimensions, a stronger statement than the one in Conjecture \ref{conjecture:mine} would be needed to account for flips and flops.
\end{remark}

\subsection{Some totally unstable rational surfaces}

The Matsushima-Lichn\'erowicz obstruction does not seem to have been fully explored in all obvious cases. There are two reasons for this: on the one hand it is not easy to describe the automorphism groups of varieties with many symmetries (i.e. those which are likely to have non-reductive automorphism groups), especially in higher dimensions. On the other hand, even when the birational class of a projective variety is well understood (e.g. rational surfaces), there may be many varieties in the class with very different automorphism groups depending on the choice of birational transformations among them. As a result, a complete classification of the cases for which Theorem \ref{theorem:matsushima-obstruction} is applicable may be hopeless. Nonetheless we can conclude this article by proving Theorem \ref{theorem:low-rank}, which can be stated in simple terms:
\begin{proof}[Proof of Theorem \ref{theorem:low-rank}]
The statement on $\mathbb F_n$ follows from Lemma \ref{lemma:hirzebruch-obstruction} or Lemma \ref{lemma:Hirzebruch}. Consider \eqref{eq:Fn-triangle}. We are interested in describing $\mathrm{Aut}^0(Y)$. By Lemma \ref{lemma:automorphism-blowup}, we are either looking for $\mathrm{Aut}^0(\mathbb F_n, p_n)$ or $\mathrm{Aut}^0(\mathbb F_{n+1}, q_{n+1})$. By plugging $p_n=([x:0:0],[u:0])$ with $x\neq 0, u\neq 0$ or $q_{n+1}=([0:0:z],[0:v])$ with $z\neq 0, v\neq 0$ in \eqref{eq:action-coordinates} as fixed points, we get that the elements of $\mathrm{Aut}^0 (\mathbb F_n,p_n)\cong\mathrm{Aut}^0(Y)\cong \mathrm{Aut}^0(\mathbb F_{n+1}, q_{n+1})$ are of the form
$$\left(a_0z_0^{n}+a_1z_1z_0^{n-1}+\cdots+a_nz_1^n, \begin{pmatrix}
	a & b\\
	0 & d
\end{pmatrix}/\mu_n\right),
$$
where $a_0,\cdots, a_n, b\in \mathbb C$, $a,b\in \mathbb C^*$, computed by considering the stabiliser of $p_n$. Hence $\mathrm{Aut}^0(Y)\cong(\mathbb G_a)^{n+1}\rtimes((\mathbb G_a\rtimes\mathbb G_m^2)/\mu_n)$, which is not reductive by Lemma \ref{lemma:reductive-trick}. The statement on the non-existence of cscK metrics follows from Theorem \ref{theorem:matsushima-obstruction}. Observe that if $\Pic(S)\leqslant 3$, then $S$ is a toric surface. Therefore, the non-existence of a cscK metric is equivalent, by the solution of the toric version of Conjecture \ref{conjecture:YTD} to toric-equivariant K-instability \cite{Donaldson-toric-surfaces-K-stability}. Hence there is an equivariant destabilising test configuration for $(S, L)$. But any equivariant test configuration is a test configuration in the sense of Definition \ref{definition:TC} and hence $(S, L)$ is K-unstable.
\end{proof}

\bibliographystyle{amsalpha}
\bibliography{bibliography}

\providecommand{\bysame}{\leavevmode\hbox to3em{\hrulefill}\thinspace}
\providecommand{\MR}{\relax\ifhmode\unskip\space\fi MR }
\providecommand{\MRhref}[2]{%
  \href{http://www.ams.org/mathscinet-getitem?mr=#1}{#2}
}
\providecommand{\href}[2]{#2}
\begin{thebibliography}{ACGT11}

\bibitem[ACGT08]{Fourauthors-Extremal-metrics-stability}
Vestislav Apostolov, David M.~J. Calderbank, Paul Gauduchon, and
  {T{\o}nnesen-Friedman, Christina W.}, \emph{Hamiltonian 2-forms in {K}\"ahler
  geometry. {III}. {E}xtremal metrics and stability}, Invent. Math.
  \textbf{173} (2008), no.~3, 547--601. \MR{2425136}

\bibitem[ACGT11]{Fourauthors-Extremal-projective-bundles}
\bysame, \emph{Extremal {K}\"ahler metrics on projective bundles over a curve},
  Adv. Math. \textbf{227} (2011), no.~6, 2385--2424. \MR{2807093}

\bibitem[AP06]{Arezzo-Pacard-blow-up}
Claudio Arezzo and Frank Pacard, \emph{Blowing up and desingularizing constant
  scalar curvature {K}\"ahler manifolds}, Acta Math. \textbf{196} (2006),
  no.~2, 179--228. \MR{2275832}

\bibitem[BB17]{Berman-Berndtsson-uniqueness-extremal}
Robert~J. Berman and Bo~Berndtsson, \emph{Convexity of the {$K$}-energy on the
  space of {K}\"ahler metrics and uniqueness of extremal metrics}, J. Amer.
  Math. Soc. \textbf{30} (2017), no.~4, 1165--1196. \MR{3671939}

\bibitem[BDL16]{Berman-Darvas-Lu-csck-implies-K-stability}
R.~J. {Berman}, T.~{Darvas}, and C.~H. {Lu}, \emph{{Regularity of weak
  minimizers of the K-energy and applications to properness and K-stability}},
  ArXiv e-prints (2016).

\bibitem[Bea96]{Beauville1996}
Arnaud Beauville, \emph{Complex algebraic surfaces}, second ed., London
  Mathematical Society Student Texts, vol.~34, Cambridge University Press,
  Cambridge, 1996, Translated from the 1978 French original by R. Barlow, with
  assistance from N. I. Shepherd-Barron and M. Reid. \MR{1406314 (97e:14045)}

\bibitem[Cal54]{Calabi-initial-problem}
Eugenio Calabi, \emph{The variation of {K}\"ahler metrics. i. the structure of
  the space; ii. a minimum problem}, Bull. Amer. Math. Soc. \textbf{60} (1954),
  no.~2, 167--168. \MR{1565563}

\bibitem[Cal82]{calabi-extremal-kahler-metrics1}
\bysame, \emph{Extremal {K}\"ahler metrics}, Seminar on {D}ifferential
  {G}eometry, Ann. of Math. Stud., vol. 102, Princeton Univ. Press, Princeton,
  N.J., 1982, pp.~259--290. \MR{645743}

\bibitem[CC18]{Chen-Cheng-calabi-dream-manifolds}
X.~{Chen} and J.~{Cheng}, \emph{{On the constant scalar curvature K\"ahler
  metrics, existence results}}, ArXiv e-prints (2018).

\bibitem[CMG18]{Cheltsov-JMG-stable-del-pezzos}
Ivan {Cheltsov} and Jesus Martinez-Garcia, \emph{{Stable polarized del Pezzo
  surfaces}}, Int. Math. Res. Not. IMRN (2018), In press.

\bibitem[CMG19]{Cheltsov-JMG-unstable-del-pezzos}
\bysame, \emph{{Unstable polarized del Pezzo surfaces}}, Trans. Amer. Math.
  Soc. (2019), In press.

\bibitem[CR18]{cheltsov-rubinstein-flops}
Ivan~A. Cheltsov and Yanir~A. Rubinstein, \emph{{On flops and canonical
  metrics}}, Ann. Sc. Norm. Super. Pisa Cl. Sci. (5) \textbf{18} (2018), no.~2,
  1--29.

\bibitem[CT08]{Chen-Tian-uniqueness-cscK}
X.~X. Chen and G.~Tian, \emph{Geometry of {K}\"ahler metrics and foliations by
  holomorphic discs}, Publ. Math. Inst. Hautes \'Etudes Sci. (2008), no.~107,
  1--107. \MR{2434691}

\bibitem[Don01]{Donaldson-uniqueness-cscK}
S.~K. Donaldson, \emph{Scalar curvature and projective embeddings. {I}}, J.
  Differential Geom. \textbf{59} (2001), no.~3, 479--522. \MR{1916953}

\bibitem[Don02]{Donaldson-scalar-curvature-K-stability-toric}
\bysame, \emph{Scalar curvature and stability of toric varieties}, J.
  Differential Geom. \textbf{62} (2002), no.~2, 289--349. \MR{1988506}

\bibitem[Don09]{Donaldson-toric-surfaces-K-stability}
Simon~K. Donaldson, \emph{Constant scalar curvature metrics on toric surfaces},
  Geom. Funct. Anal. \textbf{19} (2009), no.~1, 83--136. \MR{2507220}

\bibitem[DR17]{Darvas-Rubinstein-csck-implies-K-stability}
Tam\'as Darvas and Yanir~A. Rubinstein, \emph{Tian's properness conjectures and
  {F}insler geometry of the space of {K}\"ahler metrics}, J. Amer. Math. Soc.
  \textbf{30} (2017), no.~2, 347--387. \MR{3600039}

\bibitem[Har77]{HartshorneAG}
R.~Hartshorne, \emph{Algebraic geometry}, Springer-Verlag, New York, 1977,
  Graduate Texts in Mathematics, No. 52. \MR{MR0463157 (57 \#3116)}

\bibitem[{Lic}57]{lichnerowicz-matsushima-obstruction}
Andr\'e {Lichn\'erowicz}, \emph{{Sur les transformations analytiques des
  vari\'et\'es k\"ahl\'eriennes compactes.}}, {C. R. Acad. Sci., Paris}
  \textbf{244} (1957), 3011--3013 (French).

\bibitem[LS94]{lebrun-simanca-opennes-k-stability}
C.~LeBrun and S.~R. Simanca, \emph{Extremal {K}\"ahler metrics and complex
  deformation theory}, Geom. Funct. Anal. \textbf{4} (1994), no.~3, 298--336.
  \MR{1274118}

\bibitem[LX14]{Li-Xu-K-stability}
Chi Li and Chenyang Xu, \emph{Special test configuration and {K}-stability of
  {F}ano varieties}, Ann. of Math. (2) \textbf{180} (2014), no.~1, 197--232.
  \MR{3194814}

\bibitem[MG17]{JMG-obstructions-cscK-rational-surfaces-v1}
Jesus Martinez-Garcia, \emph{{Constant scalar curvature Kahler metrics on
  rational surfaces}}, ArXiv e-prints (2017).

\bibitem[Muk03]{mukai-book-moduli}
Shigeru Mukai, \emph{An introduction to invariants and moduli}, Cambridge
  Studies in Advanced Mathematics, vol.~81, Cambridge University Press,
  Cambridge, 2003, Translated from the 1998 and 2000 Japanese editions by W. M.
  Oxbury. \MR{2004218}

\bibitem[Oda13]{odaka-test-configurations-restriction}
Yuji Odaka, \emph{A generalization of the {R}oss-{T}homas slope theory}, Osaka
  J. Math. \textbf{50} (2013), no.~1, 171--185. \MR{3080636}

\bibitem[Ros06]{Ross-inventiones}
J.~Ross, \emph{Unstable products of smooth curves}, Invent. Math. \textbf{165}
  (2006), no.~1, 153--162. \MR{2221139}

\bibitem[RT06]{RossThomas1}
Julius Ross and Richard Thomas, \emph{An obstruction to the existence of
  constant scalar curvature {K}\"ahler metrics}, J. Differential Geom.
  \textbf{72} (2006), no.~3, 429--466. \MR{2219940 (2007c:32028)}

\bibitem[Rub14]{rubinstein-survey}
Yanir~A. Rubinstein, \emph{Smooth and singular {K}\"ahler-{E}instein metrics},
  Geometric and spectral analysis, Contemp. Math., vol. 630, Amer. Math. Soc.,
  Providence, RI, 2014, pp.~45--138. \MR{3328541}

\bibitem[Sz{\'e}07]{szekelyhidi-extremal-metrics-and-k-stability}
G\'abor Sz{\'e}kelyhidi, \emph{Extremal metrics and k-stability}, Bull. Lond.
  Math. Soc. \textbf{39} (2007), no.~1, 76--84. \MR{2303522}

\bibitem[Tia97]{TianKEimpliesAnalyticKstability}
Gang Tian, \emph{K\"ahler-{E}instein metrics with positive scalar curvature},
  Invent. Math. \textbf{130} (1997), no.~1, 1--37. \MR{1471884 (99e:53065)}

\bibitem[Wan12]{wang-height-git-weight}
Xiaowei Wang, \emph{Height and {GIT} weight}, Math. Res. Lett. \textbf{19}
  (2012), no.~4, 909--926. \MR{3008424}

\end{thebibliography}

\end{document}